\documentclass[12pt,a4paper,reqno]{amsart}
\usepackage[english]{babel}
\usepackage[applemac]{inputenc}
\usepackage[T1]{fontenc}
\usepackage{palatino}
\usepackage{amsmath}
\usepackage{amssymb}
\usepackage{amsthm}
\usepackage{amsfonts}
\usepackage{graphicx}
\usepackage[colorlinks = true, citecolor = black]{hyperref}
\author{Tuomas Orponen}
\thanks{The research was partially supported by the Academy of Finland, grant 133264 "Stochastic and harmonic analysis, interactions and applications".}
\title[Slices of self-similar sets]{On the packing measure of slices of self-similar sets}
\address{Department of Mathematics and Statistics, University of Helsinki}
\subjclass[2010]{28A80 (Primary); 28A78, 37C45 (Secondary)}
\email{tuomas.orponen@helsinki.fi}

\newcommand{\R}{\mathbb{R}}
\newcommand{\N}{\mathbb{N}}

\newcommand{\calR}{\mathcal{R}}

\newcommand{\Pd}{\dim_{\mathrm{p}}}

\newcommand{\calP}{\mathcal{P}}

\newcommand{\diam}{\operatorname{diam}}

\newcommand{\dist}{\operatorname{dist}}

\numberwithin{equation}{section}

\theoremstyle{plain}
\newtheorem{thm}[equation]{Theorem}

\newtheorem{lemma}[equation]{Lemma}

\newtheorem{cor}[equation]{Corollary}

\theoremstyle{definition}

\newtheorem{notation}[equation]{Notation}

\theoremstyle{remark}

\begin{document}

\begin{abstract} Let $K \subset \R^{2}$ be a rotation and reflection free self-similar set satisfying the strong separation condition, with dimension $\dim K = s > 1$. Intersecting $K$ with translates of a fixed line, one can study the $(s - 1)$-dimensional Hausdorff and packing measures of the generic non-empty line sections. In a recent article, T. Kempton gave a necessary and sufficient condition for the Hausdorff measures of the sections to be positive. In this paper, I consider the packing measures: it turns out that the generic section has infinite $(s - 1)$-dimensional packing measure under relatively mild assumptions.
\end{abstract}

\maketitle

\section{Introduction}

The main motivation for this paper is the article \cite{Ke} by T. Kempton, where the following question is considered. Let $s > 1$, and fix an $s$-dimensional self-similar set $K \subset \R^{d}$, satisfying the open set condition and containing no rotations or reflections. Then, pick a one-dimensional subspace $L \subset \R^{d}$, and slice $K$ with the $(d - 1)$-planes $V_{t} := \pi_{L}^{-1}\{t\}$, where $\pi_{L}$ stands for the orthogonal projection onto $L$. Under what conditions do many of the slices $K_{L,t} := K \cap V_{t}$ have positive $(s - 1)$-dimensional Hausdorff measure? The answer turns out to be closely connected with the behaviour of the projection of the measure $\mathcal{H}^{s}|_{K}$ into the line $L$, denoted by $\pi_{L\sharp}(\mathcal{H}^{s}|_{K})$. Under a mild geometric condition on the set $K$ -- implied by the strong separation condition -- Kempton proves that $\mathcal{H}^{s - 1}(K_{L,t}) > 0$ for $\pi_{L\sharp}(\mathcal{H}^{s}|_{K})$ almost all $t \in L$, if and only if $\pi_{L\sharp}(\mathcal{H}^{s}|_{K}) \ll \mathcal{H}^{1}$ with bounded density.  

By Marstrand's projection theorem, the condition $s > 1$ alone implies that $\pi_{L\sharp}(\mathcal{H}^{s}|_{K}) \ll \mathcal{H}^{1}$ for almost all one-dimensional subspaces $L$. But in most practical instances -- especially when $d = 2$ -- current methods do not shed much light on the question of whether or not $\pi_{L\sharp}(\mathcal{H}^{s}|_{K})$ has bounded density. So, it seems desirable to obtain some information about the slices $K_{L,t}$ under weaker assumptions on $\pi_{L\sharp}(\mathcal{H}^{s}|_{K})$. Since Kempton's result is a characterisation, however, such assumptions simply cannot yield information about Hausdorff measure.

In this paper, I study the \emph{packing measure} of the sets $K_{L,t}$. I restrict attention to the case $d = 2$, and, like Kempton, I only consider rotation and reflection free self-similar sets $K$ (RRFSSS in short). The main result in this setting is the following:
\begin{thm}\label{mainWeak} Let $K \subset \R^{2}$ be a RRFSSS satisfying the strong separation condition, with $\dim K = s > 1$. Let $\pi$ be the orthogonal projection onto some one-dimensional subspace. Assume that the following conditions are met.
\begin{itemize}
\item[(A)] $\pi_{\sharp}(\mathcal{H}^{s}|_{K}) \ll \mathcal{H}^{1}$.
\item[(B)] The self-similar set $\pi(K)$ has no fixed point coincidence.
\end{itemize}
Then $\Pd [K \cap \pi^{-1}\{t\}] = s - 1$ and $\calP^{s - 1}(K \cap \pi^{-1}\{t\}) = \infty$ for almost every $t \in \pi(K)$, where $\Pd$ and $\mathcal{P}^{s - 1}$ stand for packing dimension and $(s - 1)$-dimensional packing measure, respectively. 
\end{thm}

In fact, (B) can even be replaced by a slightly weaker condition, see Section \ref{mainProofs}.

To conclude the introduction, I mention another motivation for the paper, which has little to do with self-similar sets to begin with. If $K \subset \R^{d}$ is a general (Borel) set with $\mathcal{H}^{s}(K) < \infty$ for some $s > 1$, it is well-known, see \cite[Theorem 7.7]{Ma}, that $\mathcal{H}^{s - 1}(K_{L,t}) < \infty$ for almost all $t \in L$, and for every line $L$. 

For packing dimension, the closest known analogue is the following result by K. Falconer \cite[Lemma 5]{Fa2}: if $\Pd K \leq s$ for some $s > 1$, then almost all of the slices have $\Pd K_{L,t} \leq s - 1$ (in particular, the conclusion that $\Pd [K \cap \pi^{-1}\{t\}] \leq s - 1$ for almost all $t \in \pi(K)$ in Theorem \ref{mainWeak} is a corollary of this result). What Falconer's lemma does not reveal, however, is whether $\mathcal{P}^{s}(K) < \infty$ implies finite $(s - 1)$-dimensional packing measure for almost all slices -- in analogue with the situation for Hausdoff measures. Since the sets $K$ appearing in Theorem \ref{mainWeak} have $\calP^{s}(K) < \infty$, the conclusion is that the answer is definitely negative: curiously, one can find an abundance of counterexamples even in sets as regular as RRFSSS's.

\begin{cor}\label{mainCor} Let $K \subset \R^{2}$ be a RRFSSS satisfying the strong separation condition with $\dim K = s > 1$. Then, for almost all one-dimensional subspaces $L$, one has $\mathcal{P}^{s - 1}(K_{L,t}) = \infty$ for almost all $t \in \pi_{L}(K)$.
\end{cor}

\section{Acknowledgements} 

I wish to thank Tom Kempton for suggesting the problem, Katrin F\"assler for fruitful discussions, and the anonymous referees for numerous good comments. In particular, an earlier version of Theorem \ref{mainWeak} assumed that the similitudes generating $K$ are equicontractive, but then one of the reviewers generously supplied an argument to remove the extra hypothesis. 

\section{Notation and initial reductions}

\begin{notation} The Hausdorff and packing dimensions of a set $B \subset \R^{d}$ are denoted by $\dim B$ and $\Pd B$, respectively. The $s$-dimensional Hausdorff and packing measures are denoted by $\mathcal{H}^{s}$ and $\calP^{s}$. The definition of $\calP^{s}$, see \cite[\S 5.10]{Ma}, involves the concept of the \emph{$s$-dimensional packing premeasure}, $P^{s}$, along with its $\delta$-approximates $P^{s}_{\delta}$. The restriction of any measure $\mu$ on $\R^{2}$ to a $\mu$-measurable subset $B$ is denoted by $\mu|_{B}$. If $f \colon \R^{2} \to \R$ is a continuous function, the \emph{image measure} $f_{\sharp}\mu$ is a measure on $\R$ defined by $f_{\sharp}\mu(B) = \mu(f^{-1}(B))$, $B \subset \R$.

Given $A,B > 0$, I write $A \lesssim B$, if there exists an absolute constant $C \geq 1$ such that $A \leq CB$. By $A \gtrsim B$ I mean that $B \lesssim A$. The notation $A \sim B$ is used, if both $A \lesssim B$ and $A \gtrsim B$. If any of the symbols $\lesssim$, $\gtrsim$ or $\sim$ carry a parameter in the subindex, for instance $A \sim_{p} B$, then the implied constant $C$ is allowed to depend on this parameter -- and nothing else.
\end{notation}

\subsection{Self-similar sets} A non-empty compact set $K \subset \R^{d}$ is called \emph{self-similar}, if it satisfies the functional equation
\begin{equation}\label{selfSimilar} K = \bigcup_{j = 1}^{q} \psi_{j}(K), \end{equation}
where the mappings $\psi_{j}$ are contracting similitudes. This means that 
\begin{displaymath} |\psi_{j}(x) - \psi_{j}(y)| = \rho_{j}|x - y|, \qquad x,y \in \R^{d}, \end{displaymath}
where $\rho_{j} \in (0,1)$ is the \emph{contraction ratio} of the similitude $\psi_{j}$. A foundational result of Hutchinson \cite{Hu} states that to every finite family $\{\psi_{1},\ldots,\psi_{q}\}$ of contractive similitudes, there exists one and only one non-empty compact set $K$ satisfying \eqref{selfSimilar}. One often says that $K$ is \emph{generated} by the family $\{\psi_{1},\ldots,\psi_{q}\}$. In this note, I only consider \emph{rotation and reflection free} self-similar sets $K$ (RRFSSS in short). These words mean that $K$ is generated by a family of similitudes $\{\psi_{1},\ldots,\psi_{q}\}$ of the form $\psi_{j}(x) = \rho_{j} x + w_{j}$, where $\rho_{j} \in (0,1)$ and $w_{j} \in \R^{2}$. 

\subsection{Reduction from $\calP^{s - 1}$ to $P^{s - 1}$} The first step in the proof of Theorem \ref{mainWeak} is to reduce matters from the packing measure $\calP^{s - 1}$ to the packing premeasure $P^{s - 1}$, which is a larger quantity and thus easier to estimate from below. This reduction is the content of the next lemma. Before stating the lemma, let me note that in all that follows one may assume that the projection $\pi$ is the \emph{vertical projection} $\pi(x,y) = x$, and that $\pi(K) \subset [0,1]$. If $B \subset \R^{2}$ and $t \in \R$, I write
\begin{displaymath} B_{t} := B \cap \pi_{t}^{-1}\{t\}. \end{displaymath}

\begin{lemma}\label{reductionLemma} Let $K \subset \R^{2}$ be a RRFSSS generated by a family of similitudes $\{\psi_{1},\ldots,\psi_{q}\}$. Assume that $P^{s - 1}(K_{t}) = \infty$ for a.e $t \in \pi(K)$. Then $\calP^{s - 1}(K_{t}) = \infty$ for a.e. $t \in \pi(K)$.
\end{lemma}

\begin{proof} Assume that $\mathcal{H}^{1}(\pi(K)) > 0$ (otherwise the statement is vacuous). Then, let $K^{0},K^{1},K^{2},\ldots$ be an enumeration of all sets of the form $\psi_{\omega_{1}} \circ \psi_{\omega_{2}} \circ \ldots \circ \psi_{\omega_{m}}(K)$, with $m \geq 0$ and $(\omega_{1},\ldots,\omega_{m}) \in \{1,\ldots,q\}^{m}$. Associate to each $K^{j}$ the set
\begin{displaymath} E^{j} := \{t \in \pi(K^{j}) : P^{s - 1}(K^{j}_{t}) < \infty\} \subset \pi(K^{j}) \end{displaymath} 
By self-similarity, $\mathcal{H}^{1}(E^{j}) = 0$ for all $j \geq 0$. Thus, also the union $E := \bigcup_{j} E^{j} \subset \pi(K)$ has zero length. Pick $t \in \pi(K) \setminus E$. The aim is to show that $\calP^{s - 1}(K_{t}) = \infty$.

To achieve this, express $K_{t}$ as the countable union $K_{t} = \bigcup_{i} S_{i}$ of closed sets $S_{i} \subset \pi^{-1}\{t\}$. The set $K_{t}$ is compact and non-empty (as it has infinite $P^{s - 1}$-measure), so Baire's theorem states that it cannot be expressed as the countable union of closed sets without interior in the relative topology of $K_{t}$. Let $S = S_{i}$ be a set with non-empty $K_{t}$-interior. Since $K_{t} \subset K$, the relative topology of $K_{t}$ is inherited from $K$. A basis for the topology of $K$ is formed by the sets $K^{j}$, $j \geq 0$, so for any interior point $x \in S$ one may find a set $K^{j}$ such that $x \in K^{j} \cap K_{t} \subset S$. Fix such $x$ and $j$. Now, $P^{s - 1}(S) \geq P^{s - 1}(K^{j} \cap K_{t}) = P^{s - 1}(K^{j}_{t})$. The last quantity here is $\infty$, because $t = \pi(x) \in \pi(K^{j}) \setminus E^{j}$. So, $P^{s - 1}(S) = \infty$, and this means that $\calP^{s - 1}(K_{t}) = \infty$.  \end{proof}


\section{Main proofs}\label{mainProofs}

Fix the self-similar set $K$, generated by the rotation and reflection free family of similitudes $\{\psi_{1},\ldots,\psi_{q}\}$. I will abbreviate $\lesssim_{K,\pi}, \gtrsim_{K,\pi}$ and $\sim_{K,\pi}$ to $\lesssim,\gtrsim$ and $\sim$. A slightly stronger version of Theorem \ref{mainWeak} reads as follows.

\begin{thm}\label{main} The conclusion of Theorem \ref{mainWeak} remains valid, if the hypothesis \textup{(B)} is replaced by the weaker assumption
\begin{itemize}
\item[(B')] Write $a = \min \pi(K)$ and $b = \max \pi(K)$. Assume that either $\pi^{-1}\{a\}$ or $\pi^{-1}\{b\}$ meets only one of the sets $\psi_{j}(K)$.\end{itemize}
\end{thm}

To see that condition (B') is weaker than (B), first observe that $\pi(K)$ is generated by the similitudes $\psi_{j}'(t) = \rho_{j} t + \pi(w_{j})$. Then, it is easy the check that if, say, $\pi^{-1}\{\min \pi(K)\}$ meets $\psi_{i}(K)$ and $\psi_{j}(K)$, then $\min \pi(K)$ is a fixed point of both $\psi_{i}'$ and $\psi_{j}'$, and this forces $i = j$ by (B). 

The proof of Theorem \ref{main} occupies the rest of the paper. Write $\mu := \mathcal{H}^{s}|_{K}$. Then $\mu$ is a constant multiple of the the natural self-similar probability measure on $K$. In other words, $\mu$ satisfies
\begin{displaymath} \mu = \sum_{j = 1}^{q} \rho_{j}^{s} \cdot \psi_{j\sharp}\mu. \end{displaymath}

The main technical lemma of the paper, below, states that under the hypotheses (A) and (B'), $\mu$ almost all of the set $K$ can be covered with arbitrarily tall and narrow upright rectangles with the useful property that the part of $K$ inside each rectangle is contained relatively near its midpoint: 
\begin{lemma}\label{rectangles} Let $K \subset \R^{2}$ be a RRFSSS, and fix $C \geq 1$. Assuming \emph{(A)} and \emph{(B')}, the following holds for $\mu$-a.e. $x \in K$. For any $\delta > 0$, there exist concentric axes-parallel rectangles $R_{1} \subset R_{2} \subset \R^{2}$ with the following properties. 
\begin{itemize}
\item[(i)] $x \in R_{1}$, and $d(R_{2}) < \delta$. 
\item[(ii)] $h(R_{1}) \sim w(R_{1}) = w(R_{2}) \sim h(R_{2})/C$.
\item[(iii)] $K \cap R_{2} \subset R_{1}$,
\item[(iv)] $\mathcal{H}^{1}(\pi(K \cap R_{2})) \geq \eta w(R_{2})$ for some constant $\eta = \eta_{K} \in (0,1)$,
\item[(v)] $\mu(R_{2}) \sim w(R_{2})^{s}$.
\end{itemize}
Here $d,h$ and $w$ refer to diameter, height and width, respectively. The constants implicit in "$\sim$" depend only on $K$, and \textbf{not} on $C$ or $\delta$.
\end{lemma}

\begin{proof} Write
\begin{displaymath} K_{\omega_{1}\cdots\omega_{m}} := \psi_{\omega_{1}} \circ \cdots \circ \psi_{\omega_{m}}(K), \quad (\omega_{1},\ldots,\omega_{m}) \in \{1,\ldots,q\}^{m}. \end{displaymath}
Then $K_{\omega_{1}\cdots\omega_{m}}$ is the subset of $K_{\omega_{1}\cdots\omega_{m - 1}}$, which corresponds to 'using the $\omega_{m}^{th}$ rule inside $K_{\omega_{1}\cdots\omega_{m - 1}}$'. A set of the form $K_{\omega_{1}\cdots\omega_{m}}$ will be called a \emph{generation $m$ set}, and it is one of the $q$ \emph{children} of the set $K_{\omega_{1}\cdots\omega_{m - 1}}$. Grandchildren, grand grandchildren and so forth will be referred to as \emph{descendants}. 

Let $\Sigma^{\ast}$ stand for the set of finite words over the alphabet $\{1,\ldots,q\}$, and, for $r > 0$, write
\begin{displaymath} \Delta_{r} := \{(\omega_{1},\ldots,\omega_{n}) \in \Sigma^{\ast} : \rho_{\omega_{1}\cdots\omega_{n}} \leq r < \rho_{\omega_{1}\cdots\omega_{n - 1}}\}, \end{displaymath}
where $\rho_{\omega_{1}\cdots\omega_{m}} := \rho_{\omega_{1}}\cdots\rho_{\omega_{m}}$ is the contraction ratio of $\psi_{\omega_{1}} \circ \cdots \circ \psi_{\omega_{m}}$. 

Suppose for instance that (B') holds in the form that $\pi^{-1}\{a\}$ meets only one of the first generation sets $K_{j}$, say $K_{l} = \rho_{l}K + w_{l}$, $l \in \{1,\ldots,q\}$ (where $l$ stands for 'left'). Also, assume without loss of generality that $a = \min \pi(K) = 0$. Then, there exists a number $\kappa > 0$ such that $\pi^{-1}[0,\kappa]$ meets no first generation sets besides $K_{l}$, see Figure \ref{fig1}. By self-similarity, $\pi^{-1}[0,\rho_{l}\kappa]$ meets exactly one of the second generation sets, namely $K_{ll}$. In general, $\pi^{-1}[0, \rho_{l}^{k - 1}\kappa]$ meets only one of the generation $k$ sets, namely $K_{l^{k}}$ (where $l^{k}$ is shorthand for $l\cdots l$). On the other hand, $\pi^{-1}[0,\kappa]$ contains $K_{l^{N}}$ for some $N \in \N$, so $\pi^{-1}[0,\rho_{l}^{k - 1}\kappa]$ contains $K_{l^{N + k - 1}}$. The conclusions of this paragraph can be combined by writing
\begin{equation}\label{inclusions} K_{l^{N + k - 1}} \subset \pi^{-1}[0,\rho_{l}^{k - 1}\kappa] \cap K \subset K_{l^{k}}. \end{equation}

\begin{figure}[h!]
\begin{center}
\includegraphics[scale = 0.5]{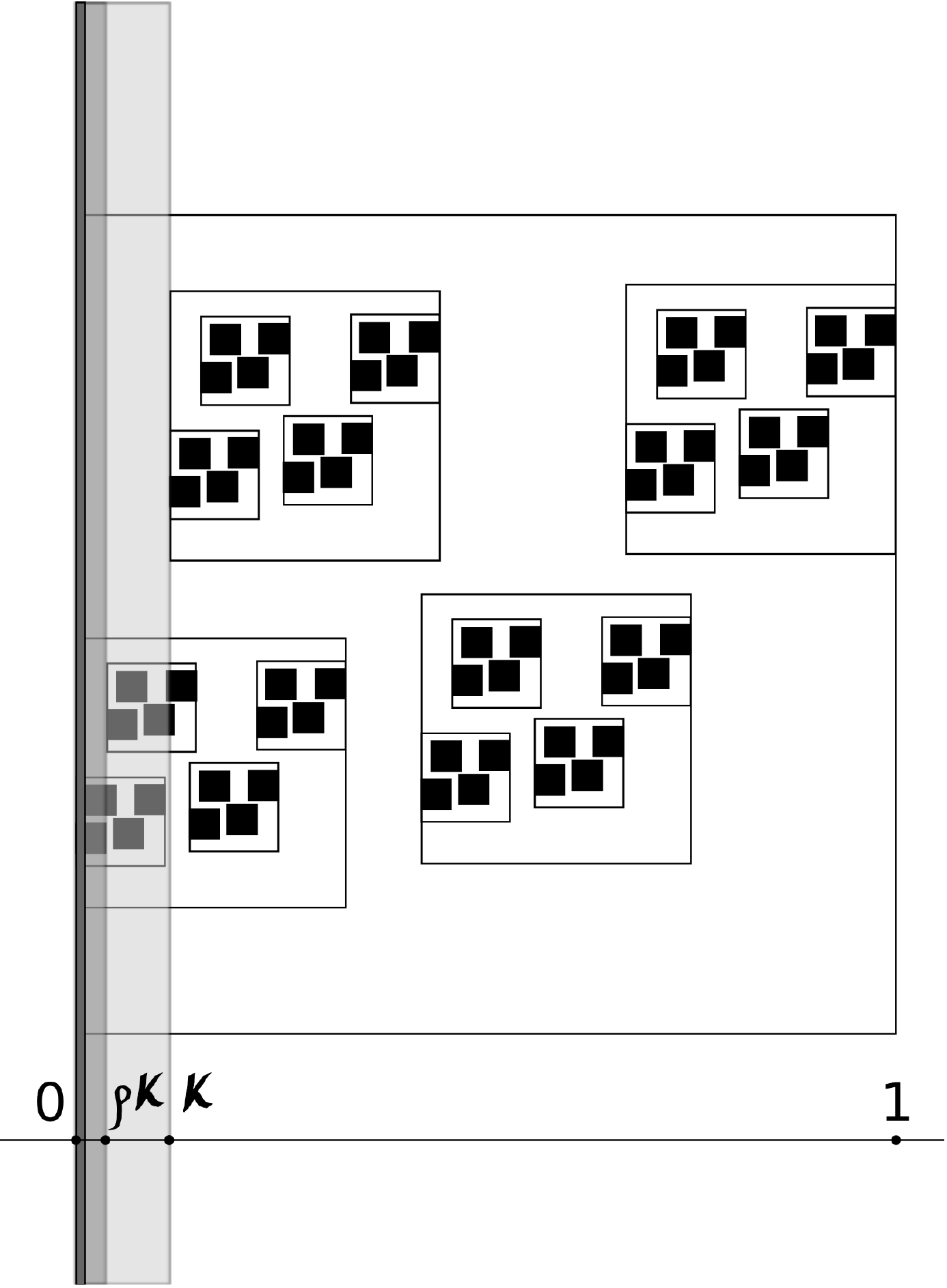}
\caption{The set $K$ and the tubes $\pi^{-1}[0,\kappa]$, $\pi^{-1}[0,\rho_{l} \kappa]$ and $\pi^{-1}[0,\rho_{l}^{2}\kappa]$.}\label{fig1}
\end{center}
\end{figure}

Now suppose that 
\begin{equation}\label{aePoint} x = (t,y) \in K_{\omega_{1}\cdots\omega_{m}l^{N + k - 1}}. \end{equation}
Here $N$ is the same number as above, and depends only on $K$. The parameter $k = k_{C} \in \N$ will be chosen large enough depending only on $C$. For fixed $k,N \in \N$, it follows by elementary probability theory that $\mu$ almost every point $x \in K$ is contained in infinitely many sets of the form \eqref{aePoint}, that is, for arbitrarily long sequences $\omega_{1}\cdots\omega_{m}$. Thus, the proof is completed by showing that the rectangles $R_{1},R_{2}$ containing $x$ and satisfying (i)--(iv) can be found, whenever \eqref{aePoint} holds.

First, observe that
\begin{equation}\label{placeOfX} x \in K_{\omega_{1}\cdots\omega_{m}l^{N + k - 1}} \subset \pi^{-1}[d,d + \rho_{\omega_{1}\cdots\omega_{m}}\rho_{l}^{k - 1}\kappa], \end{equation}
by self-similarity and \eqref{inclusions}, where $d = \min \pi(K_{\omega_{1}\cdots\omega_{m}})$. Also, for the same reasons, 
\begin{equation}\label{form2} \pi^{-1}[d,d + \rho_{\omega_{1}\cdots\omega_{m}}\rho_{l}^{k - 1}\kappa] \cap K_{\omega_{1}\cdots\omega_{m}} \subset K_{\omega_{1}\cdots\omega_{m}l^{k}}. \end{equation}

Now, define
\begin{displaymath} R_{2} = [d,d + \rho_{\omega_{1}\cdots\omega_{m}}\rho_{l}^{k - 1}\kappa] \times [y - c\rho_{\omega_{1}\cdots\omega_{m}}, y + c\rho_{\omega_{1}\cdots\omega_{m}}], \end{displaymath} 
where $c > 0$ depends only on $\rho_{\min} := \min\{\rho_{j} : 1 \leq j \leq q\} > 0$ and will be specified soon. Note that $K_{\omega_{1}\cdots\omega_{m}} \cap R_{2} \subset K_{\omega_{1}\cdots\omega_{m}l^{k}}$ by \eqref{form2}, but we need something better: the next step is to verify that  
\begin{equation}\label{inclusions2} K \cap R_{2} \subset K_{\omega_{1}\cdots\omega_{m}l^{k}}. \end{equation}
Assume that this is not the case, and find a point $z \in (K \cap R_{2}) \setminus K_{\omega_{1}\cdots\omega_{m}l^{k}}$. Let $\mathbf{i} \in \Sigma^{\ast}$ be the unique finite word in $\Delta_{\rho_{\omega_{1}\cdots\omega_{m}}}$ such that $z \in K_{\mathbf{i}}$ (observe that the sets $K_{\mathbf{i}}$, $\mathbf{i} \in \Delta_{r}$, form a partition of $K$ for any fixed $r > 0$). Then $\mathbf{i} \neq (\omega_{1}\cdots\omega_{m})$, because otherwise $z \in R_{2} \cap K_{\omega_{1}\cdots\omega_{m}} \subset K_{\omega_{1}\cdots\omega_{m}l^{k}}$ by \eqref{form2}. Now, note the general fact that if $\mathbf{i},\mathbf{j} \in \Delta_{r}$ are two distinct finite words, then $\dist(K_{\mathbf{i}},K_{\mathbf{j}}) \gtrsim r$, where the implicit constants only depend on $\rho_{\min}$ and the constants arising from the strong separation condition: with $\mathbf{i}$ as above, $\mathbf{j} = (\omega_{1}\cdots\omega_{m})$ and $r = \rho_{\omega_{1}\cdots\omega_{m}}$, this gives
\begin{displaymath} \dist(K_{\mathbf{i}},K_{\omega_{1}\cdots\omega_{m}}) \gtrsim \rho_{\omega_{1}\cdots\omega_{m}}. \end{displaymath}
Recalling that $z \in K_{\mathbf{i}} \cap R_{2}$, observing that $\diam(R_{2}) \leq 10c\rho_{\omega_{1}\cdots\omega_{m}}$, and choosing $c > 0$ small enough, it follows that $R_{2} \cap K_{\omega_{1}\cdots\omega_{m}} = \emptyset$. But this is a contradiction, since obviously $x \in R_{2} \cap K_{\omega_{1}\cdots\omega_{m}}$. Thus, \eqref{inclusions2} is proved.

Now, we claim that $R_{2}$ is the rectangle we are after, and that we can construct $R_{1} \subset R_{2}$ appropriately. Clearly (i) is satisfied, if $m$ is large enough (depending on $\delta$). Also, the ratio between $h(R_{2})$ and $w(R_{2})$ can be made to exceed $C$ by increasing $k$ (so $k$ depends only on $C$, as we promised). The rectangle $R_{1}$ is defined as the rectangle concentric with $R_{2}$, with $w(R_{1}) = w(R_{2})$ and $h(R_{1}) = A\rho_{\omega_{1}\cdots\omega_{m}}\rho_{l}^{k}$. The absolute constant $A \geq 1$ will be specified momentarily. Then (ii) is satisfied. 

To prove (iii)--(v), we choose $A$ so large that
\begin{equation}\label{inclusions3} K_{\omega_{1}\cdots\omega_{m}l^{k}} \cap \pi^{-1}[d,d + \rho_{\omega_{1}\cdots\omega_{m}}\rho_{l}^{k - 1}\kappa] \subset R_{1}. \end{equation}
Such a choice is possible, because $x = (t,y) \in K_{\omega_{1}\cdots\omega_{m}l^{k}}$, the second coordinate of the midpoint of $R_{1}$ is $y$, and the height of the set $K_{\omega_{1}\cdots\omega_{m}l^{k}}$ is 
\begin{displaymath} \lesssim \rho_{\omega_{1}\cdots\omega_{m}}\rho_{l}^{k} = h(R_{1})/A. \end{displaymath}
Then (iii) is an immediate consequence of \eqref{inclusions2} and \eqref{inclusions3}.

The claim (iv) follows from the assumption (A), which implies that $\mathcal{H}^{1}(\pi(K)) =: \tau_{K} > 0$. Combining \eqref{inclusions3} with \eqref{placeOfX}, one finds that
\begin{equation}\label{inclusions4} K_{\omega_{1}\cdots\omega_{m}l^{N + k - 1}} \subset K_{\omega_{1}\cdots\omega_{m}l^{k}} \cap \pi^{-1}[d,d + \rho_{\omega_{1}\cdots\omega_{m}}\rho_{l}^{k - 1}\kappa] \subset K \cap R_{2}, \end{equation}
which gives
\begin{displaymath} \mathcal{H}^{1}(\pi(K \cap R_{2})) \geq \mathcal{H}^{1}(\pi(K_{\omega_{1}\cdots\omega_{m}l^{N + k - 1}})) = \rho_{\omega_{1}\cdots\omega_{m}}\rho_{l}^{N + k - 1}\tau_{K} \sim_{N} \rho_{\omega_{1}\cdots\omega_{m}}\rho_{l}^{k - 1}\kappa = w(R_{2}). \end{displaymath} 
This is precisely (iv), since the constant $N$ depends only on $K$.

Finally, \eqref{inclusions4} and \eqref{inclusions2} (in this order) combined yield (v):
\begin{displaymath} w(R_{2})^{s} \sim_{N} \rho_{\omega_{1}\cdots\omega_{m}}^{s}\rho_{l}^{s(N + k - 1)} \leq \mu(R_{2}) \leq \mu(K_{\omega_{1}\cdots\omega_{m}l^{k}}) = \rho_{\omega_{1}\cdots\omega_{m}}^{s}\rho_{l}^{ks} \sim w(R_{2})^{s}. \end{displaymath}
The proof of the lemma is complete.
\end{proof}

\begin{proof}[Proof of Theorem \ref{main}] The plan is to fix any $\mathcal{H}^{1}$-positive subset $E \subset \pi(K) \subset [0,1]$, and prove that
\begin{equation}\label{integral} \int_{E} P_{\delta}^{s - 1}(K_{t}) \, dt = \infty \end{equation}
for any $\delta > 0$. This implies that $P^{s - 1}(K_{t}) = \infty$ for almost every $t \in [0,1]$: otherwise we could use Egoroff's theorem to find $\delta > 0$ and a $\mathcal{H}^{1}$-positive subset $E$ such that $P^{s - 1}_{\delta}(K_{t}) \leq A$ for $t \in E$, violating \eqref{integral}.

Fix the $\mathcal{H}^{1}$-positive subset $E \subset [0,1]$. Pick a small $\epsilon > 0$, and let $E_{0} \subset E$ be a $\mathcal{H}^{1}$-positive subset with the following property: if
\begin{itemize}
\item $I \subset \R$ is an interval of length $\ell(I) < \epsilon$, which intersects $E_{0}$, and
\item $F_{I} \subset I$ is any compact subset with $\mathcal{H}^{1}(F_{I} \cap I) \geq \eta\ell(I)$, where $\eta > 0$ is the constant from Lemma \ref{rectangles}(iii),
\end{itemize}
then 
\begin{displaymath} \mathcal{H}^{1}(E \cap F_{I}) \geq \eta\ell(I)/2. \end{displaymath}
By the Lebesgue differentiation theorem and Egoroff's theorem, such a set can be found when $\epsilon > 0$ is small enough.  

Observe that it suffices to prove \eqref{integral} for small $\delta$ (instead of all $\delta$), because $\delta \mapsto P_{\delta}^{s - 1}(K_{t})$ is a non-decreasing function. In particular, one may restrict considerations to $\delta \leq \epsilon$. Let $K_{0} \subset K$ be the set of points described in Lemma \ref{rectangles}. Then the rectangles $R_{2}$ in the said lemma (with a large parameter $C \geq 1$) form a Vitali cover for $K_{0}$, so, by the Vitali covering theorem (see \cite[Theorem 1.10]{Fa}), there exists a disjoint collection of rectangles $\calR_{2}$ such that $d(R_{2}) \leq \delta$ for all $R_{2} \in \calR_{2}$, and either
\begin{displaymath} \sum_{R_{2} \in \calR_{2}} d(R_{2})^{s} = \infty \quad \text{or} \quad \mu\left(K \setminus \bigcup_{R_{2} \in \calR_{2}} R_{2} \right) = 0. \end{displaymath}
Since $d(R_{2})^{s} \lesssim_{C} \mu(R_{2})$, the first condition is impossible by disjointness. So the second condition holds.

For $F \subset \R$, write $\calR_{F} := \{R_{2} \in \calR_{2} : K \cap R_{2} \cap \pi^{-1}(F) \neq \emptyset\}$. Then
\begin{equation}\label{form1} \pi_{\sharp}\mu(F) = \mu(\pi^{-1}(F)) \leq \sum_{R_{2} \in \calR_{F}} \mu(R_{2}). \end{equation}
Let $t \in E$. A packing\footnote{A \emph{packing} of a set $A$ is a collection of disjoint discs centred at points in $A$. These objects appear in the definition of the packing premeasure.} of the set $K_{t} = K \cap \pi^{-1}\{t\}$ can be found as follows. For each rectangle $R_{2} \in \calR_{\{t\}}$, one finds, by the definition of $\calR_{\{t\}}$, a point $x = (t,y) \in K_{t} \cap R_{2}$ such that $\pi(x) = t$, see Figure \ref{fig2}. But, since $K_{t} \cap R_{2} \subset R_{1}$ by Lemma \ref{rectangles}(iii), one actually has $(t,y) \in K_{t} \cap R_{1}$. By Lemma \ref{rectangles}(ii), $R_{1}$ is rectangle concentric with $R_{2}$, with height $h(R_{1}) \sim w(R_{2}) = h(R_{2})/C$. 
\begin{figure}[h!]
\begin{center}
\includegraphics[scale = 0.55]{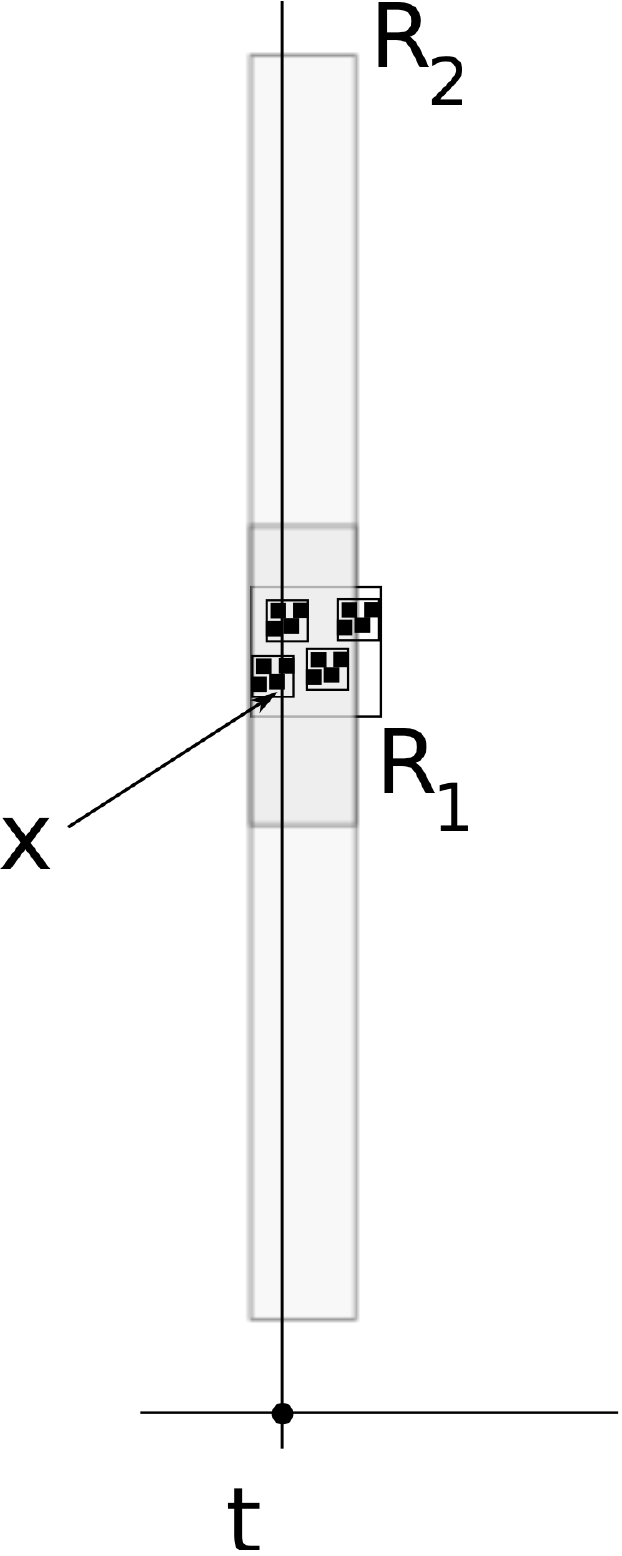}
\caption{Choosing the point $x$.}\label{fig2}
\end{center}
\end{figure}
For $C \geq 1$ large enough, one has 
\begin{displaymath} I_{R_{2}} := \{t\} \times [y - h(R_{2})/3, y + h(R_{2})/3] \subset R_{2}. \end{displaymath}
The intervals $I_{R_{2}}$ are disjoint, because the rectangles in $\calR_{\{t\}}$ are, so 
\begin{displaymath} P_{\delta}^{s - 1}(K_{t}) \geq \sum_{R_{2} \in \calR_{\{t\}}} d(I_{R_{2}})^{s - 1} \gtrsim \sum_{R_{2} \in \calR_{\{t\}}} h(R_{2})^{s - 1} = C^{s - 1}\sum_{R_{2} \in \calR_{\{t\}}} w(R_{2})^{s - 1}. \end{displaymath}
This gives
\begin{align*} \int_{E} P_{\delta}^{s - 1}(K_{t}) \, dt & \gtrsim C^{s - 1} \int_{E} \sum_{R_{2} \in \calR_{\{t\}}} w(R_{2})^{s - 1} \, dt\\
& = C^{s - 1} \sum_{R_{2} \in \calR_{2}} \int_{E \cap \pi(K \cap R_{2})} w(R_{2})^{s - 1} \, dt\\
& \geq C^{s - 1} \sum_{R_{2} \in \calR_{E_{0}}} \int_{E \cap \pi(K \cap R_{2})} w(R_{2})^{s - 1} \, dt\\
& \stackrel{(\ast)}{\gtrsim} C^{s - 1}\eta \sum_{R_{2} \in \calR_{E_{0}}} \ell(\pi(R_{2})) \cdot w(R_{2})^{s - 1}\\
& = C^{s - 1}\eta \sum_{R_{2} \in \calR_{E_{0}}} w(R_{2})^{s}\\
& \sim C^{s - 1}\eta \sum_{R_{2} \in \calR_{E_{0}}} \mu(R_{2}) \geq C^{s - 1}\eta \cdot \mu(\pi^{-1}(E_{0})) \end{align*} 
In the last inequality, \eqref{form1} was used. The $\sim$ relation on the last line is Lemma \ref{rectangles}(v). Finally, ($\ast$) follows from the definition of $E_{0}$: if $R_{2} \in \calR_{E_{0}}$, then $\pi(R_{2})$ is an interval of length $\leq \delta \leq \epsilon$ intersecting $E_{0}$, and $\pi(K \cap R_{2}) \subset \pi(R_{2})$ is a compact subset of length $\geq \eta\ell(\pi(R_{2}))$ by Lemma \ref{rectangles}(iv). Hence $\mathcal{H}^{1}(E \cap \pi(K \cap R_{2})) \geq \eta\ell(\pi(R_{2}))/2$ by the definition of $E_{0}$.

The value of the constant $C$ is independent of $\eta$ or $\pi_{\sharp}\mu(E_{0})$, so one may let $C \to \infty$. Moreover, the projected measure $\pi_{\sharp}\mu$ is equivalent to $\mathcal{H}^{1}|_{\pi(K)}$ (and not just absolutely continuous) according to a result of Peres, Schlag and Solomyak \cite[Proposition 3.1]{PSS}. This means that $\mu(\pi^{-1}(E_{0})) > 0$, so \eqref{integral} is true, and the proof is complete, by Lemma \ref{reductionLemma}. \end{proof}

\begin{proof}[Proof of Corollary \ref{mainCor}] There are only countably many one-dimensional subspaces $L$ such that the self-similar set $\pi_{L}(K)$ has exact overlaps. Since $\pi_{L\sharp}(\mathcal{H}^{s}|_{K}) \ll \mathcal{H}^{1}$ for almost all $L$ by Marstrand's projection theorem, the corollary follows directly from Theorem \ref{mainWeak}. \end{proof}

\section{An open problem}

In Theorem \ref{mainWeak}, one assumes that the projection $\pi_{\sharp}(\mathcal{H}^{s}|_{K})$ is absolutely continuous. Is this necessary? In other words, do there exist self-similar sets $K \subset \R^{2}$ such that $\dim K = s > 1$, the projection $\pi_{\sharp}(\mathcal{H}^{s}|_{K})$ is singular, and still 
\begin{displaymath} \mathcal{P}^{s - 1}(K \cap \pi^{-1}\{t\}) = \infty \end{displaymath}
for Lebesgue almost every $t \in \pi(K)$? In this case, it follows from Kempton's results \cite{Ke} that 
\begin{displaymath} \mathcal{H}^{s - 1}(K \cap \pi^{-1}\{t\}) = 0 \end{displaymath}
for $\mathcal{H}^{1}$ almost all $t \in \R$, but the packing measure is much harder to bound from above.



\end{document}